\documentclass[reqno, 11pt, letterpaper]{amsart}

\usepackage{amsmath}
\usepackage{amsfonts}
\usepackage{amssymb}
\usepackage{graphicx}
\usepackage{amsthm,graphicx,color,yfonts}
\usepackage{hyperref}

\newtheorem{theorem}{Theorem}
\newtheorem{definition}[theorem]{Definition}
\newtheorem{proposition}[theorem]{Proposition}
\newtheorem{lemma}[theorem]{Lemma}
\newtheorem{corollary}[theorem]{Corollary}

\theoremstyle{remark}
\newtheorem{remark}[theorem]{Remark}

\newcommand{\ls}{\lesssim}

\newcommand{\la}{\langle}
\newcommand{\ra}{\rangle}
\newcommand{\R}{\mathbb{R}}
\newcommand{\T}{\mathbb{T}}

\newcommand{\Z}{\mathbb{Z}}

\newcommand{\ep}{\epsilon}

\def\norm#1{\left\|#1\right\|}
\def\bra#1{\langle#1\rangle}

\usepackage{wrapfig}
\usepackage{tikz}
\usetikzlibrary{arrows,calc,decorations.pathreplacing}
\definecolor{light-gray1}{gray}{0.90}
\definecolor{light-gray2}{gray}{0.80}
\definecolor{light-gray3}{gray}{0.60}

\numberwithin{equation}{section}

\numberwithin{theorem}{section}

\numberwithin{table}{section}

\numberwithin{figure}{section}

\ifx\pdfoutput\undefined
  \DeclareGraphicsExtensions{.pstex, .eps}
\else
  \ifx\pdfoutput\relax
    \DeclareGraphicsExtensions{.pstex, .eps}
  \else
    \ifnum\pdfoutput>0
      \DeclareGraphicsExtensions{.pdf}
    \else
      \DeclareGraphicsExtensions{.pstex, .eps}
    \fi
  \fi
\fi

\title[Finite difference scheme for 2D periodic NLS]{Finite difference scheme for two-dimensional periodic nonlinear Schr\"odinger equations}


\author[Y. Hong]{Younghun Hong}
\address{Department of Mathematics, Chung-Ang University, Seoul 06974, Korea}
\email{yhhong@cau.ac.kr}

\author[C. Kwak]{Chulkwang Kwak}
\address{Facultad de Matem\'aticas, Pontificia Universidad Cat\'olica de Chile and Institute of Pure and Applied Mathematics, Chonbuk National University}
\email{chkwak@mat.uc.cl}

\author[S. Nakamura]{Shohei Nakamura}
\address{Department of Mathematics and Information Sciences, Tokyo Metropolitan University, 1-1 Minami-Ohsawa, Hachioji, Tokyo, 192-0397, Japan}
\email{nakamura-shouhei@ed.tmu.ac.jp}

\author[C. Yang]{Changhun Yang}
\address{Korea Institute for Advanced Study, Seoul 20455, Korea and Institute of Pure and Applied Mathematics, Chonbuk National University}
\email{maticionych@kias.re.kr}

\begin{document}	

\maketitle

\begin{abstract}
A nonlinear Schr\"odinger equation (NLS) on a periodic box can be discretized as a discrete nonlinear Schr\"odinger equation (DNLS) on a periodic cubic lattice, which is a system of finitely many ordinary differential equations. We show that in two spatial dimensions, solutions to the DNLS converge strongly in $L^2$ to those of the NLS as the grid size $h>0$ approaches zero. As a result, the effectiveness of the finite difference method (FDM) is justified for the two-dimensional periodic NLS.
\end{abstract}

\section{Introduction}
We consider the nonlinear Schr\"odinger equation (NLS)
\begin{equation}\label{NLS}
i\partial_t u+\Delta u-\lambda |u|^{p-1}u=0
\end{equation}
on the periodic box $\mathbb{T}^d=\mathbb{R}^d/2\pi\mathbb{Z}^d$, where $p>1$, $\lambda = \pm1$, and
$$u=u(t,x):\mathbb{R}\times \mathbb{T}^d\to\mathbb{C}.$$
The NLS is a canonical model that describes the propagation of nonlinear waves. When the nonlinearity is either cubic or quintic, or a combination of these two types, the equation \eqref{NLS} arises in various physical contexts including nonlinear optics and low-temperature physics. In particular, if a huge number of boson particles are trapped in a box with the periodic boundary condition and they are cooled to a temperature approaching absolute zero, they form a Bose--Einstein condensate and their mean-field dynamics is determined by the periodic NLS. We refer to \cite{KSS, GSS, Soh, CH} for a rigorous proof for this.

The periodic NLS \eqref{NLS} may be studied numerically by employing the following standard semi-discrete finite difference method (FDM). For a mesh size $h=\frac{\pi}{M}>0$ with a large integer $M>0$, we denote the \textit{dense} periodic lattice by
\begin{equation}\label{T_h^d}
\mathbb{T}_h^d:=h\mathbb{Z}^d/2\pi\mathbb{Z}^d,
\end{equation}
that is, the additive group
$$\Big\{x=h(m_1,...,m_d): m_j=-M, ...,-2, -1, 0, 1, 2, ..., M-1\Big\}$$
of $(2M)^d$ points (see Figure \ref{Fig:0}), and define the discrete Laplacian $\Delta_h$ by
\begin{equation}\label{discrete Laplacian}
(\Delta_h u_h)(x):=\sum_{j=1}^d\frac{u_h(x+he_j)+u_h(x-he_j)-2u_h(x)}{h^2},\quad\forall x\in\mathbb{T}_h^d,
\end{equation}
which acts on complex-valued functions on the periodic lattice. Then, we formulate the discrete nonlinear Schr\"odinger equation (DNLS) as
\begin{equation}\label{DNLS}
i\partial_t u_h + \Delta_h u_h -\lambda\left|u_h\right|^{p-1} u_h=0, 
\end{equation}
where $p>1$, $\lambda = \pm1$ and
$$u_h=u_h(t,x):\mathbb{R}\times \mathbb{T}_h^d\to\mathbb{C}.$$
In this way, the partial differential equation is translated into the system of $(2M)^d$-many first-order ordinary differential equations.

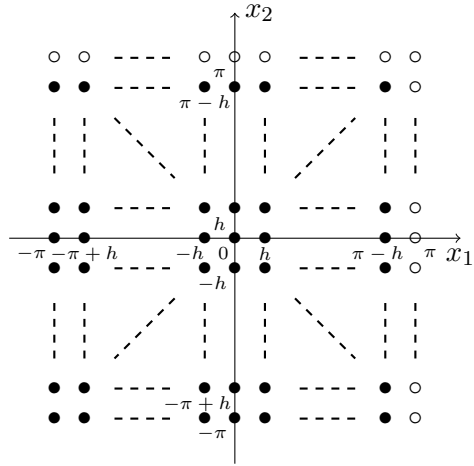
\begin{figure}[b]
\begin{center}
\begin{tikzpicture}[scale=0.5]
\draw[->] (-6,0) -- (6,0) node[below] {$x_1$};
\draw[->] (0,-6) -- (0,6) node[right] {$x_2$};
\node at (-4.8,-4.8){$\bullet$};
\node at (-4,-4.8){$\bullet$};
\node at (-0.8,-4.8){$\bullet$};
\node at (0,-4.8){$\bullet$};
\node at (0.8,-4.8){$\bullet$};
\node at (4,-4.8){$\bullet$};
\node at (4.8,-4.8){$\circ$};
\draw[thick,dashed] (-3.2,-4.8)--(-1.6,-4.8);
\draw[thick,dashed] (3.2,-4.8)--(1.6,-4.8);
\node at (-4.8,-4){$\bullet$};
\node at (-4,-4){$\bullet$};
\node at (-0.8,-4){$\bullet$};
\node at (0,-4){$\bullet$};
\node at (0.8,-4){$\bullet$};
\node at (4,-4){$\bullet$};
\node at (4.8,-4){$\circ$};
\draw[thick,dashed] (-3.2,-4)--(-1.6,-4);
\draw[thick,dashed] (3.2,-4)--(1.6,-4);
\draw[thick,dashed] (-4.8,-3.2)--(-4.8,-1.6);
\draw[thick,dashed] (-4,-3.2)--(-4,-1.6);
\draw[thick,dashed] (-1.6,-1.6)--(-3.2,-3.2);
\draw[thick,dashed] (-0.8,-3.2)--(-0.8,-1.6);
\draw[thick,dashed] (0.8,-3.2)--(0.8,-1.6);
\draw[thick,dashed] (1.6,-1.6)--(3.2,-3.2);
\draw[thick,dashed] (4,-3.2)--(4,-1.6);
\draw[thick,dashed] (4.8,-3.2)--(4.8,-1.6);
\node at (-4.8,-0.8){$\bullet$};
\node at (-4,-0.8){$\bullet$};
\node at (-0.8,-0.8){$\bullet$};
\node at (0,-0.8){$\bullet$};
\node at (0.8,-0.8){$\bullet$};
\node at (4,-0.8){$\bullet$};
\node at (4.8,-0.8){$\circ$};
\draw[thick,dashed] (-3.2,-0.8)--(-1.6,-0.8);
\draw[thick,dashed] (3.2,-0.8)--(1.6,-0.8);
\node at (-4.8,0){$\bullet$};
\node at (-4,0){$\bullet$};
\node at (-0.8,0){$\bullet$};
\node at (0,0){$\bullet$};
\node at (0.8,0){$\bullet$};
\node at (4,0){$\bullet$};
\node at (4.8,0){$\circ$};
\node at (-4.8,0.8){$\bullet$};
\node at (-4,0.8){$\bullet$};
\node at (-0.8,0.8){$\bullet$};
\node at (0,0.8){$\bullet$};
\node at (0.8,0.8){$\bullet$};
\node at (4,0.8){$\bullet$};
\node at (4.8,0.8){$\circ$};
\draw[thick,dashed] (-3.2,0.8)--(-1.6,0.8);
\draw[thick,dashed] (3.2,0.8)--(1.6,0.8);
\draw[thick,dashed] (-4.8,3.2)--(-4.8,1.6);
\draw[thick,dashed] (-4,3.2)--(-4,1.6);
\draw[thick,dashed] (-3.2,3.2)--(-1.6,1.6);
\draw[thick,dashed] (-0.8,3.2)--(-0.8,1.6);
\draw[thick,dashed] (0.8,3.2)--(0.8,1.6);
\draw[thick,dashed] (1.6,1.6)--(3.2,3.2);
\draw[thick,dashed] (4,3.2)--(4,1.6);
\draw[thick,dashed] (4.8,3.2)--(4.8,1.6);
\node at (-4.8,4){$\bullet$};
\node at (-4,4){$\bullet$};
\node at (-0.8,4){$\bullet$};
\node at (0,4){$\bullet$};
\node at (0.8,4){$\bullet$};
\node at (4,4){$\bullet$};
\node at (4.8,4){$\circ$};
\draw[thick,dashed] (-3.2,4)--(-1.6,4);
\draw[thick,dashed] (3.2,4)--(1.6,4);
\node at (-4.8,4.8){$\circ$};
\node at (-4,4.8){$\circ$};
\node at (-0.8,4.8){$\circ$};
\node at (0,4.8){$\circ$};
\node at (0.8,4.8){$\circ$};
\node at (4,4.8){$\circ$};
\node at (4.8,4.8){$\circ$};
\draw[thick,dashed] (-3.2,4.8)--(-1.6,4.8);
\draw[thick,dashed] (3.2,4.8)--(1.6,4.8);
\node at (-5.4,-0.4){{\tiny$-\pi$}};
\node at (-4,-0.4){{\tiny$-\pi+h$}};
\node at (-1.2,-0.4){{\tiny$-h$}};
\node at (0.8,-0.4){{\tiny$h$}};
\node at (3.8,-0.4){{\tiny$\pi-h$}};
\node at (5.2,-0.4){{\tiny$\pi$}};
\node at (-0.3,-0.4){{\tiny$0$}};
\node at (-0.6,-5.2){{\tiny$-\pi$}};
\node at (-1,-4.4){{\tiny$-\pi+h$}};
\node at (-0.6,-1.2){{\tiny$-h$}};
\node at (-0.4,0.4){{\tiny$h$}};
\node at (-0.8,3.6){{\tiny$\pi-h$}};
\node at (-0.4,4.4){{\tiny$\pi$}};
\end{tikzpicture}
\end{center}
\caption{Representation of the two-dimensional periodic lattice $\T_h^2$. All points marked by $\bullet$ are contained in $\T_h^2$, whereas $(4M+1)$-points marked by $\circ$ are excluded. }\label{Fig:0}
\end{figure}

The main purpose of the work presented in this article is to justify the effectiveness of the above numerical scheme. We introduce the following operators to formulate the problem precisely.

\begin{definition}[Discretization and linear interpolation]\label{def:Discrete} (i) For a function $f:\mathbb{T}^d\to\mathbb{C}$, its discretization is defined by 
\begin{equation}\label{discretization}
(d_hf)(x):=\frac{1}{h^d}\int_{x+[0,h)^d} f(y) dy,\quad \forall x\in \mathbb{T}_h^d.
\end{equation}
(ii) Given a function $f_h: \mathbb{T}_h^d\to\mathbb{C}$, its linear interpolation is defined by
\begin{equation}\label{p_h}
(p_hf_h)(x):=f_h(\underline{x})+\sum_{j=1}^dD_{h,j}^+(\underline{x}) (x_j-\underline{x}_j)
\end{equation}
for $x\in \underline{x}+[0,h)^d$ with $\underline{x}\in\mathbb{T}_h^d$, where $e_j$ is the $j$-th standard unit vector, $x_j$ denotes the $j$-th component of $x\in\mathbb{T}^d$, and
$D_{h,j}^+$ is the discrete right-hand side derivative on $\T_h^d$, i.e.,
\begin{equation}\label{discrete derivative}
D_{h,j}^+f_h(\underline{x}):=\frac{f_h(\underline{x}+he_j)-f_h(\underline{x})}{h}.
\end{equation}
\end{definition}

\begin{definition}[Nonlinear propagators]
We denote the nonlinear propagator for NLS \eqref{NLS} by $U(t)$. In other words, $U(t)u_0$ is the solution to the NLS \eqref{NLS} with initial data $u_0$. Similarly, we denote the nonlinear propagator for DNLS \eqref{DNLS} by $U_h(t)$.
\end{definition}

\begin{remark}
$(i)$ The discretization operator $d_h$ sends functions on the periodic box to functions on a periodic lattice. Conversely, the linear interpolation operator $p_h$ maps discrete functions to continuous functions.\\
$(ii)$ The nonlinear propagators $U(t)$ and $U_h(t)$ are well defined  because the equations are locally well posed under suitable assumptions (see Propositions \ref{WP:NLS} and \ref{GWP}).
\end{remark}

Our goal is then to show that
\begin{equation}\label{continuum limit}
\left(p_h\circ U_h(t)\circ d_h\right)u_0-U(t)u_0\to 0
\end{equation}
as $h\to 0$ in a proper sense (see Figure \ref{Fig:1}). The convergence \eqref{continuum limit} is referred to as the \textit{continuum limit} for DNLS. Obviously, proving the continuum limit implies the effectiveness of the numerical scheme.

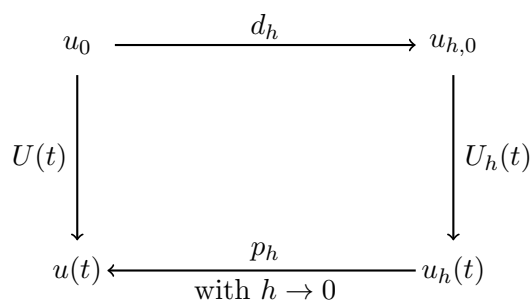
\begin{figure}[b]
\begin{center}
\begin{tikzpicture}[scale=0.5]
\draw[thick,->] (-4,3) -- (4,3);
\node at (0,3.5){$d_h$};
\draw[thick,->] (5,2.2) -- (5,-2.2);
\node at (6.2,0){$U_h(t)$};
\draw[thick,->] (4,-3) -- (-4.2,-3);
\node at (0,-2.5){$p_h$};
\node at (0,-3.5){with $h \to 0$};
\draw[thick,->] (-5,2.2) -- (-5,-2.2);
\node at (-6,0){$U(t)$};
\node at (-5,3){$u_0$};
\node at (5,3){$u_{h,0}$};
\node at (5,-3){$u_h(t)$};
\node at (-5,-3){$u(t)$};
\end{tikzpicture}
\end{center}
\caption{Schematic representation of the \emph{continuum limit} for DNLS. The nonlinear propagator $U(t)$ maps initial data $u_0$ to the solution $u(t)$ to NLS \eqref{NLS}. 
On the other hand, discretizing initial data $u_0$ to $u_{h,0}$ and then acting the nonlinear propagator $U_h(t)$ on $u_{h,0}$ enables generating a solution $u_h(t)$ to DNLS \eqref{DNLS}. Theorem \ref{main theorem} asserts that its linear interpolation $(p_hu_h)(t)$ converges to $u(t)$ as $h\to 0$.}\label{Fig:1}
\end{figure}
Despite its physical importance, to the best of the authors' knowledge, the continuum limit for a nonlinear dispersive equation on a compact manifold has not previously been studied. However, the continuum limit from DNLS on $h\mathbb{Z}^d$ to NLS on $\mathbb{R}^d$ has been investigated by Ignat--Zuazua \cite{IZ-05-1, IZ-05-2,IZ-09,IZ-12}, Kirkpatrick--Lenzmann--Staffilani \cite{KLS-13}, and the first and fourth authors of this work \cite{HY1,HY2}. An important remark is that the linear discrete model
\begin{equation}\label{discrete LS0}
i\partial_t u_h+\Delta_h u_h=0,
\end{equation}
where $u_h=u_h(t,x):\mathbb{R}\times h\mathbb{Z}^d\to\mathbb{C}$,
enjoys \textit{weaker dispersion} than the continuous model \cite{IZ-05-1, SK-05}, and this causes difficulties in proving the continuum limit. In \cite{IZ-05-2, IZ-12}, the authors circumvented the weak dispersion phenomena by introducing a new numerical scheme, that is, the \textit{two-grid algorithm}, to exclude bad frequencies generating weak dispersions. Subsequently, in \cite{HY1,HY2}, the authors discovered that the space--time norm bounds, namely Strichartz estimates, for \eqref{discrete LS0} hold uniformly in $h\in(0,1]$ with some derivative on the right-hand side. As an application, convergence of the discrete NLS on $h\mathbb{Z}^d$ is established without modifying the numerical scheme. 

Returning to the problem discussed here, one would attempt to adopt the approach in \cite{HY1,HY2} to the periodic setting. However, several new issues are raised, in particular, for the desired uniform-in-$h$ Strichartz estimates.

\begin{remark}\label{why difficult}
In the celebrated work of Bourgain \cite{B-93}, Strichartz estimates are established for the linear Schr\"odinger equation $i\partial_t u+\Delta u=0$ on a periodic box, and they are applied to prove the local well-posedness of the periodic NLS \eqref{NLS} in a low regularity Sobolev space. Importantly, these Strichartz estimates can be captured from the gain of regularity in the multi-interaction of linear solutions localized in same frequencies but different modulations. This phenomenon is known as the \emph{dispersive smoothing effect}, and its proof requires an understanding of the geometry of the support of the spacetime Fourier transform of linear solutions, that is, the hypersurface $\{(\tau, k)\in\mathbb{R}\times\mathbb{Z}^d : \tau + |k|^2 = 0\}$. We also refer to the work of Guo, Oh and Wang \cite{GuoOhWang} for a further context of NLS on irrational torus.

Unfortunately, we are currently unable to capture dispersive smoothing in the discrete setting. Indeed, the hypersurface for the  linear equation \eqref{discrete LS0} is given by $\{(\tau, k)\in\mathbb{R}\times\mathbb{Z}^d : \tau + \sum_{j=1}^d\frac{2}{h^2}(1-\cos hk_j) = 0\}$. Following Bourgain's approach, it is necessary to count the maximal number of points in the intersection of twisted annuli $\{\tilde{k}\in\mathbb{Z}^d: M \le |\tilde{\tau} +\sum_{j=1}^d \frac{2}{h^2}(1-\cos h\tilde{k}_j) | \le 2M\}$ and $\{ \tilde{k}'\in\mathbb{Z}^d: N  \le |\tilde{\tau}' + \sum_{j=1}^d\frac{2}{h^2}(1-\cos h\tilde{k}'_j) | \le 2N\}$ restricted to the hyperplane $\tilde{k} + \tilde{k}' =k$ with $\tilde{\tau}+\tilde{\tau}'=\tau$. Compared to the continuous case, the situation is much more complicated because of the complexity of the geometry. Moreover, because local smoothing is known to fail on the noncompact lattice $h\mathbb{Z}^d$ \cite{IZ-05-1}, this may not simply be a matter of technicality but may indicate that a new idea is needed. It is also worth to mention that Strichartz estimates on $\mathbb{T}^d$ for higher dimension were established by Bourgain and Demeter \cite{BD} as a corollary of their main theorem on the decoupling inequality (Wolff's inequality). It may be one of possible ways to follow the decoupling approach to our problem. We leave this question for future study.
\end{remark}

One way to circumvent the aforementioned difficulties would be to approximate the linear propagator on a periodic box by that on an entire space. Ultimately it would seem that, by suitably adjusting the argument of Vega \cite{V-92} to the discrete setting, the time-localized uniform-in-$h$ Strichartz estimates can be obtained on a periodic lattice. For the statement, we define the finite-dimensional vector space $L_h^r=L_h^r(\mathbb{T}_h^d)$ equipped with the norm
\begin{equation}\label{L_h^r}
\|f\|_{L_h^r}
:=\left\{\begin{aligned}
&\bigg\{h^d\sum_{x\in \mathbb{T}_h^d}|f(x)|^r\bigg\}^{1/r}&&\textup{if }1\leq r<\infty,\\
&\sup_{x\in \mathbb{T}_h^d}|f(x)|&&\textup{if }r=\infty,
\end{aligned}\right.
\end{equation}
and define the fractional derivative $\langle\nabla_h\rangle^s$ as the Fourier multiplier of symbol $\langle k\rangle^s$ via the discrete Fourier transform, where $\langle k\rangle=\sqrt{1+|k|^2}$ (see Section 2). We say that $(q,r)$ is \textit{lattice-admissible} if $2\leq q,r\leq\infty$,
\begin{equation}\label{r-admissible}
\frac{3}{q}+\frac{d}{r}=\frac{d}{2}\textup{ and }(q,r,d)\neq (2,\infty,3).
\end{equation}

\begin{theorem}[Strichartz estimates on a periodic lattice]\label{Strichartz} 
Let $h\in(0,1]$. For a lattice-admissible pair $(q,r)$, there exists $C>0$, independent of $h$, such that
\begin{equation}\label{Strichartz estimate 1}
\left\| e^{it\Delta_h}u_{h,0}\right\|_{L_t^q([0,1]; L_h^r)}\leq C \|\langle\nabla_h\rangle^{\frac{2}{q}+\epsilon}u_{h,0} \|_{L_h^2}
\end{equation}
for any $\ep >0$.
\end{theorem}

Strichartz estimates are one of the fundamental tools to study dispersive equations because they quantify the smoothing and/or decay properties of solutions. On the unbounded lattice $\mathbb{Z}^d$, the smoothing and decay properties have been investigated for various models: 
we refer to \cite{SK-05, IZ-05-1, Ignat} for the Schr\"odinger equation, \cite{Schultz} for the wave equation, and \cite{BG} for the Klein--Gordon equation. Theorem \ref{Strichartz} is the first result on a compact discrete domain as far as the authors know. It should be noted that the inequality \eqref{Strichartz estimate 1} holds uniformly in $h\in(0,1]$. 
Indeed, it is easy to show the inequality $\| e^{it\Delta_h}u_{h,0}\|_{L_t^q([0,1]; L_h^r)}\leq C_h\|u_{h,0}\|_{L_h^2}$ for all $1 \le q,r \le \infty$, since $\mathbb{T}_h^d$ is finite-dimensional. However, this inequality is not useful at all for our purpose because the constant $C_h$ blows up as $h\to 0$. 
As in \cite{HY1}, for which uniform Strichartz estimates are proven on $h\mathbb{Z}^d$, we could obtain an appropriate (uniform-in-$h$) Strichartz estimates by placing some derivatives on the norm on the right-hand side (Theorem \ref{Strichartz}). We also note that we do not claim optimality of the Strichartz estimates \eqref{Strichartz estimate 1}. In fact, the order of the derivative could be reduced by solving the counting problem mentioned in Remark \ref{why difficult}.


Although there is still room for improvement, Theorem \ref{Strichartz} is sufficient to establish the global-in-time continuum limit for the two-dimensional periodic NLS, which is the main result of this work.

\begin{theorem}[Continuum limits]\label{main theorem}
Let $d=2$. We assume
	\begin{equation}\label{assumption 1}
	\left\{
	\begin{aligned}
	1&<p<\infty&& \textup{when }\lambda =1&& \textup{(defocusing)},\\
	1&<p<3&&\textup{when }\lambda = -1&& \textup{(focusing)}.
	\end{aligned}
	\right.
	\end{equation}
There exist constants $A,B>0$, independent of $h\in(0,1]$, such that for all $t\in\mathbb{R}$,
$$\left\|\left(p_h\circ U_h(t)\circ d_h\right)u_0-U(t)u_0\right\|_{L^2(\mathbb{T}^2)}\leq A \sqrt{h}e^{B|t|}\left(1+\|u_{0}\|_{H^1 (\mathbb{T}^2)}\right)^p.$$
\end{theorem}

The proof of Theorem \ref{main theorem} follows the argument outlined in \cite{HY2}. Precisely, we consider two solutions in Duhamel's formulas,
$$u_h(t)=e^{-it(-\Delta_h)}(d_hu_0)-i\lambda\int_0^t e^{-i(t-s)(-\Delta_h)}(|u_h|^{p-1}u_h)(s)ds$$
and
$$u(t)=e^{-it(-\Delta)}u_0-i\lambda\int_0^t e^{-i(t-s)(-\Delta)}(|u|^{p-1}u)(s)ds,$$
where $u_h(t)=U_h(t)(d_hu_0)$ and $u(t)=U(t)u_0$. We aim to estimate the difference $p_hu_h(t)-u(t)$ directly by the standard Gr\"onwall's inequality. We accomplish this by making use of a ``time-averaged'' uniform-in-$h$ $L_h^\infty$-bound for nonlinear solutions $\{u_h(t)\}_{h\in(0,1]}$. Such a uniform bound can be obtained by applying uniform-in-$h$ Strichartz estimates for the discrete linear equation to the nonlinear problem.

\begin{remark}
$(i)$ The essential part of our analysis lies in proving the uniform Strichartz estimates for the linear equation. For this proof, we employ the Fourier analysis on a periodic lattice, and we develop harmonic analysis tools on the lattice, including the Littlewood--Paley theory. Indeed, a periodic lattice is a finite abelian group; thus, the Fourier and its inverse transforms are properly defined (see Section \ref{sec: Fourier transform}).\\
$(ii)$ As mentioned in Remark \ref{why difficult}, if the classical Bourgain's argument is adopted, the proof of the Strichartz estimates is transferred to a certain counting problem, but this is ultimately quite challenging. Instead, we employ an alternative approach of Vega \cite{V-92}. This approach is simpler and can also be applied to more general settings \cite{BGT-04}, but optimality is far from guaranteed. \\
$(iii)$ In higher dimensions $d\geq 3$, only local-in-time convergence can be derived from Theorem \ref{Strichartz}, because uniform Strichartz estimates hold for more regular initial data than those in the energy space. Indeed, if $d\geq 3$, the regularity $\frac{2}{q}+\epsilon$ of the Sobolev norm on the right-hand side in \eqref{Strichartz estimate 1} is always strictly greater than one (when $r = \infty$). 
\end{remark}

The remainder of the paper is organized as follows:  In Section \ref{sec: prelim}, we provide the collection of basic analysis tools. In particular, Fourier analysis on a periodic lattice is briefly presented, but some important inequalities, such as the Sobolev and the Gagliardo--Nirenberg inequalites, are also introduced. In Section \ref{sec: Strichartz proof}, we prove the key uniform Strichartz estimates (Theorem \ref{Strichartz}). In Section \ref{sec:GWP}, we establish a well-posedness theory for DNLS \eqref{DNLS} as well as uniform bounds for the nonlinear solutions. Finally, in Section \ref{sec: proof of main theorem}, we prove the main theorem (Theorem \ref{main theorem}).

\subsection*{Acknowledgement}
Y.H. was supported by the Basic Science Research Program through the National Research Foundation of Korea (NRF) funded by the Ministry of Education (NRF-2017R1C1B1008215). C.K. was supported by FONDECYT Postdoctorado 2017 Proyect No 3170067. S. N. was supported by the JSPS Grant-in-Aid for JSPS Research Fellow no. 17J01766. 
C.Y. was supported by the Samsung Science and Technology Foundation under Project Number SSTF-BA1702-02.

\section{Preliminaries}\label{sec: prelim}

\subsection{Basic inequalities on a periodic lattice}

Recall the definition of the Lebesgue spaces on a periodic lattice (see \eqref{L_h^r}). On a lattice, we often have a larger class of inequalities, compared to those in the continuum domain $\mathbb{T}^d$. For instance, by the definition, one can easily show the inequality
\begin{equation}\label{reversed L^p-L^q inequality}
\|u\|_{L_h^q}\lesssim h^{-(\frac{1}{p}-\frac{1}{q})}\|u\|_{L_h^p}\quad\textup{for all }q> p,
\end{equation}
while the embedding $L^p\hookrightarrow L^q$ fails on $\mathbb{T}^d$. However, these inequalities become meaningless in the continuum limit $h\to 0$. Therefore, we would have to use inequalities wherein the implicit constants are independent of $h\in(0,1]$.

We state the following inequalities, which hold uniformly in $h\in(0,1]$.
\begin{lemma}
$(i)$ (H\"older's inequality).
If $\frac1p+\frac1q=\frac{1}{r}$ and $1\le p,q, r\le \infty$, then 
\begin{equation}\label{ineq:holder}
\| uv\|_{L_h^r} \leq \| u\|_{L_h^p}\| v\|_{L_h^q}.
\end{equation}
$(ii)$ (Young's inequality)
If $\frac1p+\frac1q=\frac1r+1$, $1\le p,q,r \le \infty$ and $\frac1p+\frac1q\ge1$, then 
\begin{equation}\label{eq:young's}
\| u* v\|_{L_h^r} \leq \| u\|_{L_h^p}\|v\|_{L_h^q},
\end{equation}
where $*$ denotes the convolution operator defined by
\begin{equation}\label{eq:convolution}
(u*v)(x) = h^d \sum_{y\in \T_h^d} u(x-y)v(y).
\end{equation}
\end{lemma}

\begin{proof}
Based on H\"older's and Young's inequalities for sequences, we prove that
$$\| uv\|_{L_h^r}=h^{\frac{d}{r}}\|uv\|_{\ell_x^r} \leq h^{\frac{d}{r}}\| u\|_{\ell_x^p}\| v\|_{\ell_x^q}=\| u\|_{L_h^p}\| v\|_{L_h^q}$$
and
$$\| u* v\|_{L_h^r}=h^{d(1+\frac{1}{r})}\Big\|\sum_{y\in\mathbb{T}_h^d} u(x-y)v(y)\Big\|_{\ell_x^r} \leq h^{d(\frac{1}{p}+\frac{1}{q})} \| u\|_{\ell_x^p}\|v\|_{\ell_x^q}=\| u\|_{L_h^p}\|v\|_{L_h^q}.$$
\end{proof}

\subsection{Fourier transform on a periodic lattice}\label{sec: Fourier transform}
Fix a large integer $M>0$. For the periodic lattice $\mathbb{T}_h^d$ with $h=\frac{\pi}{M}$ (see \eqref{T_h^d}), we denote its Fourier dual space, that is, the \textit{sparse} periodic lattice, by
\begin{equation}\label{eq:FDS}
\begin{aligned}
(\mathbb{T}_h^d)^*:&=\mathbb{Z}^d/\tfrac{2\pi}{h}\mathbb{Z}^d=(\mathbb{Z}/\tfrac{2\pi}{h}\mathbb{Z})^d\\
&=\Big\{-\tfrac{\pi}{h}, ..., -2, -1,0,1, 2, ..., \tfrac{\pi}{h}-1\Big\}^d\\
&=\Big\{-M, ..., -2, -1,0,1, 2, ..., M-1\Big\}^d.
\end{aligned}
\end{equation}
For a function $u:\mathbb{T}_h^d\to\mathbb{C}$, its Fourier transform $\mathcal{F}_hu: (\mathbb{T}_h^d)^*\to\mathbb{C}$ is defined by
$$(\mathcal{F}_hu)(k):=h^d\sum_{x\in \mathbb{T}_h^d} u(x)e^{-i k\cdot x}.$$
The inverse Fourier transform of a function $u:(\mathbb{T}_h^d)^*\to\mathbb{C}$ is given by
$$(\mathcal{F}_h^{-1}u)(x):=\frac{1}{(2\pi)^d}\sum_{k\in (\mathbb{T}_h^d)^*} u(k)e^{i k\cdot x}.$$
With abuse of notation, we write $\sum_{x\in \mathbb{T}_h^d}=\sum_x$ and $\sum_{k\in (\mathbb{T}_h^d)^*}=\sum_k$ unless there is confusion.

\begin{remark}
The above definitions are consistent with those on the periodic box $\mathbb{T}^d$. Indeed, formally, we have
$$\mathbb{T}_h^d\to\mathbb{T}^d,\quad(\mathbb{T}_h^d)^*\to\mathbb{Z}^d,\quad\mathcal{F}_h\to \mathcal{F},\quad\mathcal{F}_h^{-1}\to\mathcal{F}^{-1}$$
as $h\to0$, where $\mathcal{F}$ and $\mathcal{F}^{-1}$ are the Fourier and the inverse transforms on $\mathbb{T}^d$, respectively, 
$$(\mathcal{F}u)(k):=\int_{\mathbb{T}^d} u(x)e^{-i k\cdot x}dx,\quad (\mathcal{F}^{-1}u)(x):=\frac{1}{(2\pi)^d}\sum_{k\in \mathbb{Z}^d} u(k)e^{i k\cdot x}.$$
\end{remark}

We collect the properties of the Fourier and inverse Fourier transforms.
\begin{lemma}[Properties of the Fourier transform on a periodic lattice]\label{prelim properties}\ 
\begin{enumerate}
\item (Inversion)
$$\mathcal{F}_h^{-1}\circ \mathcal{F}_h=\textup{Id}\textup{ on }L^2(\mathbb{T}_h^d),\quad\mathcal{F}_h\circ \mathcal{F}_h^{-1}=\textup{Id}\textup{ on }L^2((\mathbb{T}_h^d)^*).$$
\item (Plancherel's theorem) $$\frac{1}{(2\pi)^d}\sum_{k}(\mathcal{F}_hu)(k)\overline{(\mathcal{F}_hv)(k)}=h^d\sum_{x}u(x)\bar{v}(x).$$
\item (Fourier transform of a product)
$$\mathcal{F}_h(uv)(k)=\frac{1}{(2\pi)^d}\sum_{k'}(\mathcal{F}_hu)(k')(\mathcal{F}_hv)(k-k').$$
\end{enumerate}
\end{lemma}
To prove Lemma \ref{prelim properties}, we need the following identities.
\begin{lemma}\label{useful summation lemma}
\begin{equation}\label{eq:USL1}
\frac{h^d}{(2\pi)^d}\sum_{k}e^{i k\cdot x}=\delta(x):=\left\{\begin{aligned}
&1&&\textup{if }x=0,\\
&0&&\textup{if }x\neq 0
\end{aligned}\right.
\end{equation}
and
\begin{equation}\label{eq:USL2}
\frac{h^d}{(2\pi)^d}\sum_{x}e^{i k\cdot x}=\delta(k).
\end{equation}

\end{lemma}

\begin{proof}
We only prove \eqref{eq:USL1}, because the proof of \eqref{eq:USL2} is similar. Recalling that $h=\frac{\pi}{M}$ and $x=(x_1,..., x_d)=(hm_1, ..., hm_d)\in \mathbb{T}_h^d$ where $m_j\in\{-M, -M+1,..., M-2, M-1\}$, we evaluate the geometric sum
$$\begin{aligned}
\sum_{k_j=-M}^{M-1}e^{i k_jx_j}&=\left\{\begin{aligned}
&2M=\frac{2\pi}{h} &&\textup{if }x_j=0,\\
&\frac{e^{-i Mx_j}(e^{2iM x_j}-1)}{e^{i x_j}-1}=\frac{e^{-i Mx_j}(e^{2\pi i m_j }-1)}{e^{i x_j}-1}=0 &&\textup{if }x_j\in\mathbb{T}_h\setminus\{0\}
\end{aligned}\right.\\
&=\frac{2\pi}{h}\delta(x_j).
\end{aligned}$$
Thus, we conclude that
$$\sum_{k}e^{i k\cdot x}=\prod_{j=1}^d\sum_{k_j=-M}^{M-1}e^{i k_jx_j}=\prod_{j=1}^d\frac{2\pi}{h}\delta(x_j)=\frac{(2\pi)^d}{h^d}\delta(x),$$
where we use the fact that $(\mathbb{T}_h^d)^*=\{-M, -M+1,..., M-1\}^d$
\end{proof}

\begin{proof}[Proof of Lemma \ref{prelim properties}]
(1) A direct calculation in addition to \eqref{eq:USL1} yields
\begin{align*}
\left(\mathcal{F}_h^{-1}(\mathcal{F}_hu)\right)(x)&=\frac{1}{(2\pi)^d}\sum_k \left\{h^d\sum_{x'} u(x')e^{-i k\cdot x'}\right\}e^{i k\cdot x}\\
&= \sum_{x'}\left\{\frac{h^d}{(2\pi)^d} \sum_{k}e^{i k\cdot (x-x')}\right\} u(x')\\
&= \sum_{x'}\delta(x-x') u(x')=u(x).
\end{align*}
Analogously, one can show that $\left(\mathcal{F}_h(\mathcal{F}_h^{-1}u)\right)(k)=u(k)$.\\
(2) Similarly, using \eqref{eq:USL1}, we prove that
\begin{align*}
\frac{1}{(2\pi)^d}\sum_{k}(\mathcal{F}_hu)(k)\overline{(\mathcal{F}_hv)(k)}&=\frac{1}{(2\pi)^d}\sum_{k}\left\{h^d\sum_{x} u(x)e^{- i k\cdot x}\right\}\left\{h^d\sum_{x'} \overline{v(x')}e^{i k\cdot x'}\right\}\\
&=h^d\sum_{x}\sum_{x'} u(x)\overline{v(x')} \left\{\frac{h^d}{(2\pi)^d}\sum_{k}e^{-i k\cdot (x-x')}\right\}\\
&=h^d\sum_{x}u(x)\overline{v(x)}.
\end{align*}
(3) We write
$$\begin{aligned}
\mathcal{F}_h(uv)(k)&=h^d\sum_{x} u(x)v(x) e^{-ik\cdot x}\\
&=h^d\sum_{x} \left\{\frac{1}{(2\pi)^d}\sum_{\ell}(\mathcal{F}_hu)(\ell)e^{i\ell\cdot x}\right\}\left\{\frac{1}{(2\pi)^d}\sum_{\ell'}(\mathcal{F}_hv)(\ell')e^{i\ell'\cdot x}\right\}e^{-ik\cdot x}\\
&=\frac{1}{(2\pi)^d}\sum_{\ell}\sum_{\ell'}(\mathcal{F}_hu)(\ell)(\mathcal{F}_hv)(\ell') \left\{\frac{h^d}{(2\pi)^d}\sum_{x}e^{i(\ell+\ell'-k)\cdot x}\right\}.
\end{aligned}$$
Then, applying \eqref{eq:USL2} and summing out $\ell'$, we prove the desired identity.
\end{proof}

By the Fourier transform, we see that the discrete Laplacian is a Fourier multiplier operator. 
\begin{lemma}[Discrete Laplacian as a Fourier multiplier operator]\label{Laplacian as a multiplier}
The discrete Laplacian $-\Delta_h$ is the Fourier multiplier of the symbol $\sum_{j=1}^d\frac{4}{h^2}\sin^2(\frac{hk_j}{2})=\sum_{j=1}^d\frac{2}{h^2}(1-\cos hk_j)$.
\end{lemma}

\begin{proof}
By the definition \eqref{discrete Laplacian}, 
$$\mathcal{F}_h\left((-\Delta_h) u\right)(k)=\sum_{j=1}^d\frac{2-e^{ih k_j}-e^{-ih k_j}}{h^2}(\mathcal{F}_hu)(k)=\sum_{j=1}^d\frac{2(1-\cos hk_j)}{h^2}(\mathcal{F}_hu)(k).$$
\end{proof}

\begin{remark}
The discrete Laplacian formally converges to the Laplacian on $\mathbb{T}^d$ as $h\to 0$, because given $k\in(\mathbb{T}_h^d)^*$, the multiplier for the discrete Laplacian converges to that for the Laplacian on $\mathbb{T}^d$, i.e., $\sum_{j=1}^d\frac{2}{h^2}(1-\cos(hk_j))\to |k|^2$ as $h\to0$.
\end{remark}

\subsection{Dyadic decompositions and Sobolev spaces}\label{sec: Littlewood-Paley}
Let
$$N_*=2^{\ell_*}\quad\textup{with}\quad\ell_*=\lceil\log_2(\tfrac{h}{\pi})\rceil -1,$$
%
%
%
%
%
where $\lceil a\rceil$ denotes the smallest integer greater than or equal to $a$. For a dyadic number $N = 2^\ell$ with $\ell\in\mathbb{Z}$ such that $N_*\leq N\leq 1$, we define the frequency projection operator $P_N=P_N^h$ by
\begin{equation}\label{LP projection}
(P_N u)(x) := \left\{\begin{aligned}
&\frac{1}{(2\pi)^d}\sum_{\frac{\pi N}{2h}<\max|k_j|\le \frac {\pi N}{h}}
(\mathcal F_hu) (k)e^{i k\cdot x }&&\textup{if}\quad 2N_*\le N\leq 1,\\
&\frac{1}{(2\pi)^d}(\mathcal F_hu) (0) &&\textup{if}\quad N=N_*.
\end{aligned}\right.
\end{equation}
For $s\in\mathbb{R}$, we define the Sobolev space $H_h^s$ by the Hilbert space equipped with the norm
\begin{equation}\label{eq:Hs}
\|u\|_{H_h^s}:=\left\{\frac{1}{(2\pi)^d}\sum_{k}\langle k\rangle^{2s}\left|(\mathcal{F}_hu)(k)\right|^2\right\}^{1/2}.
\end{equation}
We observe that 
$$\|u\|_{H_h^s}^2 \sim \sum_{N_* \le N \le 1} \left\langle\frac{N}{h}\right\rangle^{2s}\|P_N f\|_{L_h^2}^2.$$

The following Sobolev and Gagliardo--Nirenberg inequalities are used in our analysis.

\begin{lemma}[Sobolev embedding]\label{Sobolev inequality}
Suppose that $0<s\leq\frac{d}{2}$, $q\geq 2$ and $\frac{1}{q}=\frac{1}{2}-\frac{s}{d}$. Then, for any $\epsilon>0$, we have
\begin{equation}\label{eq: Sobolev inequality}
\|u\|_{L_h^q}\lesssim \|u\|_{H_h^{s+\epsilon}}.
\end{equation}
\end{lemma}

\begin{lemma}[Gagliardo--Nirenberg inequality]\label{Lem:GN}
Suppose $\frac1p=\frac12-\frac{\theta}{d}$, $1<p\le \infty$ and $0<\theta<1$. Then we have
\begin{align*}
\| f\|_{L_h^p} \ls\|f\|_{L_h^2}^{1-\theta}  \| f\|_{H_h^1}^{\theta}.
\end{align*}
\end{lemma}

Proofs of Lemmas \ref{Sobolev inequality} and \ref{Lem:GN} are given in Appendix \ref{sec: appendix Sobolev}. We expect the inequality \eqref{eq: Sobolev inequality} to be improved to the sharp version $(\epsilon=0)$ by adopting the argument in \cite{BO} for instance. Nevertheless, in this study, we employ a nonsharp version, because its proof is simpler but also sufficient for our analysis.

\subsection{Norm equivalence}
There are several ways to define Sobolev spaces on a periodic lattice. The following lemma shows that the Sobolev norm defined by \eqref{eq:Hs} is equivalent to that by the discrete derivatives \eqref{discrete derivative} as well as that by $\sqrt{1-\Delta_h}$, that is, the Fourier multiplier of the symbol $(1+\sum_{j=1}^{d} \frac{4}{h^2}\sin^2(\frac{hk_j}{2}))^{1/2}$.

\begin{lemma}[Norm equivalence]\label{Lem:normequivalence}
$$\|u\|_{H_h^1}\sim \|\sqrt{1-\Delta_h}u\|_{L_h^2}=\left\{\|u\|_{L_h^2}^2+\| D_{h}^+ u\|_{L_h^2}^2\right\}^{1/2},$$
where $D_h^+=(D_{h,1}^+, ..., D_{h,d}^+)$.
\end{lemma}
\begin{proof}
The first equivalence follows from the Plancherel theorem and the pointwise bound $(1+\sum_{j=1}^{d} \frac{4}{h^2}\sin^2(\frac{hk_j}{2}))^{1/2}\sim\sqrt{1+|k|^2}$ on $(\mathbb{T}_h^d)^*$. The second identity follows from $(D_h)^*D_h=-\Delta_h$.
\end{proof}

%
%
%
%

\section{Uniform Strichartz estimates on a periodic lattice}\label{sec: Strichartz proof}

This section is devoted to the proof of our key uniform-in-$h$ Strichartz estimates (Theorem \ref{Strichartz}). First, we reduce the proof to the following dispersive estimate. 
\begin{proposition}[Dispersive estimate]\label{dispersion estimate}
Let $h\in(0,1]$. For any dyadic number $N$ with $N_*:=2^{\lceil\log_2(\tfrac{h}{\pi})\rceil-1}\leq N\leq 1$,  there exists $c>0$ such that if $|t|\le \frac{ch}{N}$, then
\begin{equation}\label{Linftybound:Dyadic}
\| e^{it \Delta_h } P_{\leq N}u_{h,0}\|_{L^\infty_h} \ls  \left(\frac{N}{h|t|}\right)^{\frac d3}\|u_{h,0} \|_{L_h^1},
\end{equation}
where
$$(P_{\leq N}u_{h,0})(x):=\frac{1}{(2\pi)^d}\sum_{\max|k_j|\leq \frac{\pi N}{h}}(\mathcal{F}_hu_{h,0})(k)e^{ik\cdot x}.$$
\end{proposition}

\begin{proof}[Proof of Theorem \ref{Strichartz}, assuming Proposition \ref{dispersion estimate}]
Applying the standard interpolation argument of Keel and Tao \cite{KT-98} with the dispersive estimate \eqref{Linftybound:Dyadic} but restricting this to the time interval $[0,\frac{ch}{N}]$, one can prove that 
$$\| e^{it \Delta_h } P_{\leq N}u_{h,0}\|_{L_t^q([0,\frac{ch}{N}];L^r_h)} \ls  \left(\frac{N}{h}\right)^{\frac 1q}\|u_{h,0} \|_{L_h^2}.$$
Hence, by changing the variables in time, $P_N=P_{\leq N}P_N$ and the unitarity of the Schr\"odinger flow, we obtain
$$\begin{aligned}
\| e^{it \Delta_h } P_Nu_{h,0}\|_{L_t^q([\frac{ch(n-1)}{N},\frac{chn}{N}];L^r_h)}&=\| e^{i(t+\frac{ch(n-1)}{N}) \Delta_h } P_Nu_{h,0}\|_{L_t^q([0,\frac{ch}{N}];L^r_h)}\\
&=\| e^{it \Delta_h } P_{\leq N}(P_N e^{i\frac{ch(n-1)}{N} \Delta_h } u_{h,0})\|_{L_t^q([0,\frac{ch}{N}];L^r_h)}\\
&\ls  \left(\frac{N}{h}\right)^{\frac 1q}\|P_N e^{i\frac{ch(n-1)}{N} \Delta_h } u_{h,0}\|_{L_h^2}\\
&= \left(\frac{N}{h}\right)^{\frac 1q}\|P_Nu_{h,0} \|_{L_h^2}.
\end{aligned}$$
Summing in the time interval, 
$$\begin{aligned}
\|e^{it \Delta_h } P_Nu_{h,0}\|_{L_t^q([0,1];L^r_h)}^{ q}&\leq\sum_{n=1}^{\lceil\frac{N}{ch}\rceil}\| e^{it \Delta_h } P_Nu_{h,0}\|_{L_t^q([\frac{ch(n-1)}{N},\frac{chn}{N}];L^r_h)}^{q}\\
&\lesssim\sum_{n=1}^{\lceil\frac{N}{ch}\rceil}\frac{N}{h}\|P_Nu_{h,0} \|_{L_h^2}^{q}=\left(\frac{N}{h}\right)^{2}\|P_Nu_{h,0} \|_{L_h^2}^{q}.
\end{aligned}$$
Then, summing in $N$, we obtain
$$\begin{aligned}
\| e^{it \Delta_h } u_{h,0}\|_{L_t^q([0,1];L^r_h)}&=\left\|\sum_{N=N_*}^1 e^{it \Delta_h } P_Nu_{h,0}\right\|_{L_t^q([0,1];L^r_h)}\\
&\leq\sum_{N=N_*}^1\| e^{it \Delta_h } P_Nu_{h,0}\|_{L_t^q([0,1];L^r_h)}\\
&\ls\sum_{N=N_*}^1  \left(\frac{N}{h}\right)^{\frac 2q}\|P_Nu_{h,0} \|_{L_h^2}.
\end{aligned}$$
Because $\mathcal{F}_h(P_Nu_{h,0})$ is localized in $|k|\sim\frac{N}{h}$, we conclude that
$$\| e^{it \Delta_h } u_{h,0}\|_{L_t^q([0,1];L^r_h)}\ls\sum_{N=N_*}^1  \left(\frac{N}{h}\right)^{-\epsilon}\|u_{h,0} \|_{H_h^{\frac{2}{q}+\epsilon}}\lesssim \left(\frac{N_*}{h}\right)^{-\epsilon}\|u_{h,0} \|_{H_h^{\frac{2}{q}+\epsilon}}\sim \|u_{h,0} \|_{H_h^{\frac{2}{q}+\epsilon}},$$
where in the last step, we used that $N_*\sim 2^{\log_2(\frac{h}{\pi})}=\frac{h}{\pi}$.
\end{proof}

%

Proposition \ref{dispersion estimate} remains to be proved, for which we need to estimate the sums of the oscillating functions. Following Vega's argument \cite{V-92}, we use Lemma \ref{Lem:Riemannsum} to approximate the sums by the oscillatory integrals. Then, we employ the estimate (Lemma \ref{lem:VC}) of the oscillatory integral.

\begin{lemma}[Zygmund $\textup{\cite[Chapter V, Lemma 4.4]{Zygmund}}$]\label{Lem:Riemannsum}
Let $\varphi$ be a real-valued function, and let $a, b \in \mathbb R$ with $a < b$. If $\varphi'$ is monotonic and $|\varphi'| < 2\pi$ on $(a,b)$, then
\[\left | \int_a^b e^{i \varphi(x)} dx - \sum_{a < n \le b} e^{i \varphi(n)} \right| \le A,\]
where the constant $A$ is independent of $\varphi$, $a$, and $b$. 
\end{lemma}


\begin{lemma}\label{lem:VC}
Let $h \in (0,1]$ and a dyadic number $N$ with $N_*\le N \le 1$ be given. We define 
$$I_{N,h,t,x}:=\int_{-\frac{\pi N}{h}}^{\frac{\pi N}{h}} e^{i(x\xi-\frac{2t}{h^2}\left(1-\cos h\xi\right))} d\xi.$$
Then, there exists $B>0$, independent of $h$ and $N$, such that 
\[|I_{N,h,t,x}| \le B \left(\frac{N}{h|t|}\right)^\frac13.\]
\end{lemma}

\begin{proof}
If $N_*\le N \le \frac{1}{4}$, by the van der Corput lemma with $|(x\xi-\tfrac{2t}{h^2}(1-\cos h\xi))''|=2|t||\cos h\xi|\geq|t|$ for $|\xi|\leq\frac{\pi N}{h}$, we have $|I_{N,h,t,x}| \lesssim|t|^{-1/2}$. Hence, interpolating with the trivial bound $|I_{N,h,t,x}| \lesssim\frac{N}{h}$, we obtain the desired bound.

If $N=\frac{1}{2}$ or $1$, then we decompose 
$$I_{N,h,t,x}=I_{\frac{1}{4}, h, t,x}+\int_{\frac{\pi}{4h}\leq |\xi|\leq\frac{\pi N}{h}} e^{i ( x\xi-\frac{2t}{h^2}\left(1-\cos h\xi\right))} d\xi.$$
It has already been shown that $|I_{\frac{1}{4}, h, t,x}|\lesssim (h|t|)^{-1/3}$. For the integral on the right-hand side, we change the variables,
$$\int_{\frac{\pi}{4h}\leq |\xi|\leq\frac{\pi N}{h}} e^{i ( x\xi-\frac{2t}{h^2} \left(1-\cos h\xi\right))} d\xi=h^{-1}\int_{\frac{\pi}{4}\leq |\xi|\leq\pi N} e^{i ( \frac{x}{h}\xi-\frac{2t}{h^2} \left(1-\cos\xi\right))} d\xi.$$
We observe that 
$|( \frac{x}{h}\xi-\frac{2t}{h^2}\left(1-\cos\xi\right))''|=\frac{2|t|}{h^2}|\cos\xi|$ and  $|( \frac{x}{h}\xi-\frac{2t}{h^2} \left(1-\cos\xi\right))'''|=\frac{2|t|}{h^2}|\sin\xi|$. Thus, applying the van der Corput lemma again, we prove that 
$$\left|\int_{\frac{\pi}{4h}\leq |\xi|\leq\frac{\pi N}{h}} e^{i ( x\xi-\frac{2t}{h^2} \left(1-\cos h\xi \right))} d\xi\right|\lesssim h^{-1}\frac{1}{(|t|/h^2)^{1/3}}=\frac{1}{(h|t|)^{1/3}}.$$
Therefore, we complete the proof of the lemma.
\end{proof}


\begin{proof}[Proof of Proposition \ref{dispersion estimate}]
We consider the case $N=N_*$. By the definition of $P_{N_*}$ (see \eqref{LP projection}) and the Plancherel theorem, we have
$$e^{it \Delta_h } P_{N_*}u_{h,0}(x) =\frac{1}{(2\pi)^d}(\mathcal{F}_h u_{h,0})(0)\lesssim\left\{\frac{1}{(2\pi)^d}\sum_{k}|\mathcal{F}_h u_{h,0}(k)|^2\right\}^{\frac12}=\|u_{h,0}\|_{L_h^2},$$
which implies $\| e^{it \Delta_h } P_{N_*}u_{h,0} \|_{L^\infty_h} \lesssim \|u_{h,0}\|_{L_h^2}$.
Hence, interpolating it with a trivial estimate $\| e^{it \Delta_h } P_{N_*}u_{h,0} \|_{L^2_h} \leq \| u_{h,0} \|_{L_h^2}$, we 
get the bound $\|e^{it \Delta_h } P_{N_*} u_{h,0}\|_{L_h^r} \lesssim \|u_{h,0}\|_{L_h^2}$ for all $r\geq 2$. As a consequence, we obtain $\|e^{it \Delta_h } P_{N_*}u_{h,0}\|_{L_t^q([0,1];L_h^r)}\leq \|e^{it \Delta_h } P_{N_*} u_{h,0}\|_{L_t^\infty([0,1];L_h^r)} \lesssim \|P_{N_*}u_{h,0}\|_{L_h^2}$ for any $1 \le q \le \infty$ and $2 \le r \le \infty$.

Suppose that $2N_*\leq N\leq 1$. A direct calculation with Lemma \ref{Laplacian as a multiplier} yields
\[\begin{aligned}
e^{it \Delta_h } P_{\leq N}u_{h,0} (x)&= \frac{1}{(2\pi)^d}\sum_{\max |k_j|\le \frac{\pi N}{h}}  e^{i( k\cdot x - \sum_{j=1}^d\frac{2t}{h^2}\left(1-\cos hk_j \right))} \mathcal F_h( P_Nu_{h,0})(k) \\
&=\frac{h^d}{(2\pi)^d} \sum_{x'}  P_Nu_{h,0}(x')\prod_{j=1}^{d} \sum_{|k_j|\le \frac{\pi N}{h}} e^{i( (x_j-x'_j) k_j - \frac{2t}{h^2}\left(1-\cos hk_j\right))} \\
&= (K_{N,t} \ast  P_Nu_{h,0})(x),
\end{aligned}\]
where
$$K_{N,t}(x)= \frac{1}{(2\pi)^d}\prod_{j=1}^{d} \sum_{|k_j|\le \frac{\pi N}{h}} e^{i( x_jk_j - \frac{2t}{h^2}\left(1-\cos hk_j\right))}$$
and $\ast$ is the convolution on the lattice defined in \eqref{eq:convolution}. Hence, Young's inequality ensures \eqref{Linftybound:Dyadic} provided the following one-dimensional inequality holds true: 
\begin{equation}\label{eq:1d}
\sup_{x\in\mathbb{T}_h}\left|  \sum_{|k|\le \frac{\pi N}{h}} e^{i \left(k x - \frac{2t}{h^2}\left(1-\cos hk\right)\right)} \right| \ls \left(\frac{N}{h|t|}\right)^\frac13.
\end{equation}


It remains to prove \eqref{eq:1d}. For notational convenience, we write
$$\varphi(\xi) :=x\xi-\frac{2t}{h^2}\left(1-\cos h\xi \right)$$
for $x \in \mathbb T_h$ and $\xi\in\mathbb{R}$ with $|\xi|\le \tfrac{\pi N}{h}$. A direct calculation yields
\[|\varphi'(\xi)| =\left| x - \frac{2t}{h}\sin h\xi\right| < 2\pi\]
under the restriction $|t|\le \frac{ch}{N}$.

First, we consider the case $N\le\frac12$.
Then $\varphi'$ is decreasing on $[-\frac{\pi N}{h},\frac{\pi N}{h}]$.
Hence, applying Lemma \ref{Lem:Riemannsum} and \ref{lem:VC}, we obtain
\[\begin{aligned}
\left| \sum_{|k|\le \frac{\pi N}{h}} e^{i\varphi(k)} \right| \le&~{} \left| \sum_{|k|\le \frac{\pi N}{h}} e^{i\varphi(k)} - \int_{-\frac{\pi N}{h} }^{\frac{\pi N}{h} } e^{ i\varphi(\xi)} \; d\xi \right|  + \left| \int_{-\frac{\pi N}{h} }^{\frac{\pi N}{h} } e^{ i\varphi(\xi) }\; d\xi \right| \\
\le&~{} A+ B \left(\frac{N}{h|t|}\right)^\frac13 \lesssim  \left(\frac{N}{h|t|}\right)^\frac13.
\end{aligned}\]
In the last inequality, we used $|t|\le \frac{ch}{N}$ and $N\geq N_*\sim\frac{h}{\pi}$, implying $1\lesssim(\frac{h}{N|t|})^{1/3}\lesssim(\frac{N}{h|t|})^{1/3}$.
Next, we consider the case $N=1$. We divide the interval into three parts:
$$ [ -\tfrac{\pi}{h},\tfrac{\pi}{h} ) = [-\tfrac{\pi}{h},-\tfrac{\pi}{2h}] \cup [-\tfrac{\pi}{2h},\tfrac{\pi}{2h}] \cup [\tfrac{\pi}{2h},\tfrac{\pi}{h})=:I_1\cup I_2\cup I_3,$$
where $\varphi'$ is monotonic on each $I_j$.
Then, we decompose
\begin{align*}
\sum_{|k|\le \frac{\pi }{h}} e^{i\varphi(k)}=\sum_{k\in I_1} e^{i\varphi(k)}+\sum_{k\in I_2} e^{i\varphi(k)}+\sum_{k\in I_3} e^{i\varphi(k)} =: S_1+S_2+S_3,
\end{align*}
Each $S_j$ can be estimated with the same method as above. Summing these, we complete the proof.
\end{proof}

As an application of Theorem \ref{Strichartz}, we obtain the time-averaged uniform $L_h^{\infty}$ estimates.
\begin{corollary}[Uniform time-averaged $L_h^\infty$-bounds for the discrete linear Schr\"odinger flow; 2D case]\label{linear L^infty bound}
Suppose that $d=2$ and $1 \le q < \infty$.
Then, 
\begin{equation}\label{eq:Lq}
\|e^{it\Delta_h}u_{h,0}\|_{L_t^{q}([0,1];L_h^\infty)}\lesssim \|u_{h,0}\|_{H_h^1}.
\end{equation}
\end{corollary}

\begin{proof}
Let $\epsilon>0$ be a sufficiently small number such that the following inequalities hold. 

For $1 \le q \le 3$, H\"older's inequality in time and Theorem \ref{Strichartz} yield
\[\|e^{it\Delta_h}u_{h,0}\|_{L_t^{q}([0,1]; L_h^{\infty})} \leq \|e^{it\Delta_h}u_{h,0}\|_{L_t^3([0,1];L_h^{\infty})} \lesssim\|u_{h,0}\|_{H_h^{\frac{2}{3}+\epsilon}}.\]
Suppose that $q > 3$. By the Sobolev inequality (Lemma \ref{Sobolev inequality}) and the unitarity of the Schr\"odinger flow, we get
\[\|e^{it\Delta_h}u_{h,0}\|_{L_t^{\infty}([0,1];L_h^{\infty})} \lesssim\|e^{it\Delta_h}u_{h,0}\|_{L_t^{\infty}([0,1];H_h^{1+\epsilon})}= \|u_{h,0}\|_{H_h^{1+\ep}},\]
for a small $\epsilon = \epsilon(q)>0$ appeared in Theorem \ref{Strichartz}. Thus, interpolating this inequality and Theorem \ref{Strichartz} with $(q,r,d) = (3,\infty,2)$ and choosing $\epsilon < \frac1q$, we obtain
\[\|e^{it\Delta_h}u_{h,0}\|_{L_t^{q}([0,1];L_h^{\infty})} \lesssim \|u_{h,0}\|_{H_h^{1-\frac1q+\ep}}\leq \|u_{h,0}\|_{H_h^1},\]
which completes the proof. 
\end{proof}

\section{Uniform bound for discrete NLS}\label{sec:GWP}

In this section, we provide a simple well-posedness theorem for DNLS \eqref{DNLS}. Then, as an application of the uniform-in-$h$ Strichartz estimates, we deduce a uniform time-averaged $L_h^\infty$-bound for nonlinear solutions.

\subsection{Global well-posedness}

By Duhamel's principle, DNLS \eqref{DNLS} is equivalent to the integral equation
\begin{equation}\label{eq:Duhamel}
u_h(t)=e^{it\Delta_h}u_{h,0}-i\lambda \int_0^t e^{i(t-s)\Delta_h}(|u_h|^{p-1}u_h)(s) \; ds.
\end{equation}
We next show its global well-posedness.

\begin{proposition}[Global well-posedness]\label{GWP}
Let $d \ge 1$, $h > 0$ and $p>1$. Then, for any initial data $u_{h,0}\in L_h^2$, there exists a unique global solution $u_h(t)\in C(\mathbb{R}; L_h^2)$ to DNLS \eqref{DNLS}. Moreover, it conserves the mass
\begin{equation}\label{eq:M}
M_h(u_h):=\|u_h\|_{L_h^2}^2
\end{equation}
and the energy
\begin{equation}\label{eq:E}
E_h(u_h):=\frac{1}{2}\|\sqrt{-\Delta_h}u_h\|_{L_h^2}^2+\frac{\lambda}{p+1}\|u_h\|_{L_h^{p+1}}^{p+1}
\end{equation}
\end{proposition}

\begin{proof}
The proof is identical to the analogous theorem for the discrete NLS on $h\Z^d$ (see \cite[Proposition 6.1]{HY1}). Fix $h>0$. For a small $T>0$ to be chosen later, let $X_T:=C_t([-T,T];L_h^2)$. We denote by $\Gamma(u_h)$ the right-hand side of \eqref{eq:Duhamel}. Then, by the unitarity of the linear propagator and the trivial inequality $\|u_h\|_{L_h^\infty}\leq h^{-d/2}\|u_h\|_{L_h^2}$, one can show that 
$$\begin{aligned}
\|\Gamma(u_h)\|_{X_T}&\leq \|u_{h,0}\|_{L^2}+\||u_h|^{p-1}u_h\|_{L_t^1([-T,T];L_h^2)}\\
&\leq \|u_{h,0}\|_{L^2}+T\|u_h\|_{C_t([-T,T];L_h^{\infty})}^{p-1}\|u_h\|_{C_t([-T,T];L_h^2)}\\
& \leq \|u_{h,0}\|_{L^2}+Th^{-\frac{d(p-1)}{2}}\|u_h\|_{X_T}^p
\end{aligned}$$
and in the same way, 
$$\|\Gamma(u_h)-\Gamma(v_h)\|_{X_T}\lesssim Th^{-\frac{d(p-1)}{2}}\left\{\|u_h\|_{X_T}^{p-1}+\|v_h\|_{X_T}^{p-1}\right\}\|u_h-v_h\|_{X_T}.$$
Therefore, if $T>0$ is sufficiently small depending on $h>0$, $\Gamma$ is contractive on the set $\{u_h\in X_T: \|u_h\|_{X_T}\leq 2\|u_{h,0}\|_{L_h^2}\}$. Thus, DNLS \eqref{DNLS} is locally well-posed in $L_h^2$. The conservation laws \eqref{eq:M} and \eqref{eq:E} can be proven by direct calculations. The lifespan of local solutions is then extended by the mass conservation law \eqref{eq:M}.
\end{proof}

\subsection{Uniform bound for the 2D DNLS}
Next, we show that not only linear solutions (Corollary \ref{linear L^infty bound}) but also nonlinear solutions obey a time-averaged uniform $L_h^\infty$-bound. 

\begin{proposition}[Uniform $L_h^\infty$-bound for the 2D DNLS]\label{L^infty bound}
Suppose that $p$ satisfies \eqref{assumption 1}. Then, the solution $u_h(t)$ to DNLS  \eqref{DNLS} with initial data $u_{h,0}\in H_h^1$, constructed in Proposition \ref{GWP}, satisfies
\begin{equation}\label{long time L^infty bound}
\|u_h\|_{L_t^{q_*}([-T,T]; L_h^\infty)}\lesssim \langle T\rangle^{1/q_*}\|u_{h,0}\|_{H_h^1}, \ \forall T>0,
\end{equation}
where 
$$\left\{\begin{aligned}
&q_*>p-1 &&\textup{if }p\geq 3,\\
&q_*=2 &&\textup{if }1<p<3.
\end{aligned}\right.$$
\end{proposition}

\begin{proof}
Let $u_h$ be the solution to DNLS \eqref{DNLS} constructed in Proposition \ref{GWP}, and let $\tau>0$ be a sufficiently small number such that 
$$\|u_h\|_{C_t(I;H_h^1)} +  \|u_h\|_{L_t^{q_*}(I;L_h^{\infty})} \le 4c_0\|u_{h,0}\|_{H_h^1},$$
where $I=[-\tau, \tau]$, $c_0 = \max(c_{q_*},1)$ and $c_{q_*}$ is the implicit constant in \eqref{eq:Lq} (when $q = q_*$). Such $\tau$ is initially chosen depending on $h>0$, but later it can be extended independently of $h>0$. 

From the integral representation of the solution \eqref{eq:Duhamel}, the unitarity of the linear flow yields
\begin{equation}\label{L^infty bound proof 1}
\|u_h\|_{C_t(I; H_h^1)} \le \|u_{h,0}\|_{H_h^1} + \||u_h|^{p-1}u_h\|_{L_t^1(I;H_h^1)},
\end{equation}
and by Corollary \ref{linear L^infty bound}, we obtain
\begin{equation}\label{L^infty bound proof 2}
\|u_h\|_{L_t^{q_*}(I; L_h^{\infty})} \le c_0\|u_{h,0}\|_{H_h^1} + c_0\||u_h|^{p-1}u_h\|_{L_t^1(I;H_h^1)}.
\end{equation}
Applying the fundamental theorem of calculus of the form
\begin{equation}\label{FTC}
\begin{aligned}
|\alpha|^{p-1}\alpha-|\beta|^{p-1}\beta&=\int_0^1 \frac{d}{ds}\left\{\left|\alpha+s(\beta-\alpha)\right|^{p-1}\left(\alpha+s(\beta-\alpha)\right)\right\}ds\\
&=\frac{p+1}{2}\int_0^1 \left|\alpha+s(\beta-\alpha)\right|^{p-1}ds\cdot (\beta-\alpha)\\
&\quad+\frac{p-1}{2}\int_0^1 \left|\alpha+s(\beta-\alpha)\right|^{p-3}(\alpha+s(\beta-\alpha))^2ds\cdot \overline{\beta-\alpha}
\end{aligned}
\end{equation}
with $\alpha=u_h(x+he_j)$ and $\beta=u_h(x)$, we obtain
$$\begin{aligned}
\|D_{h;j}^+(|u_h|^{p-1}u_h)\|_{L_h^2}&=\frac{1}{h}\|(|u_h|^{p-1}u_h)(x+he_j)-|u_h|^{p-1}u_h(x)\|_{L_h^2}\\
&\lesssim \frac{1}{h}\|u_h\|_{L_h^\infty}^{p-1}\|u_h(x+he_j)-u_h(x)\|_{L_h^2}=\|u_h\|_{L_h^\infty}^{p-1}\|D_{h;j}^+u_h\|_{L_h^2}.
\end{aligned}$$
Hence, by the norm equivalence (Lemma \ref{Lem:normequivalence}), it follows that 
$$\begin{aligned}
\||u_h|^{p-1}u_h\|_{H_h^1}&\sim\||u_h|^{p-1}u_h\|_{L_h^2}+\sum_{j=1}^d\|D_{h;j}^+(|u_h|^{p-1}u_h)\|_{L_h^2}\\
&\lesssim \|u_h\|_{L_h^\infty}^{p-1}\|u_h\|_{L_h^2}+\sum_{j=1}^d \|u_h\|_{L_h^\infty}^{p-1}\|D_{h;j}^+u_h\|_{L_h^2}\\
&\sim \|u_h\|_{L_h^\infty}^{p-1}\|u_h\|_{H_h^1}.
\end{aligned}$$
Inserting this bound in \eqref{L^infty bound proof 1} and \eqref{L^infty bound proof 2}, we obtain
$$\begin{aligned}
&\|u_h\|_{C_t(I; H_h^1)}+\|u_h\|_{L_t^{q_*}(I; L_h^{\infty})}\\
&\le 2c_0\|u_{h,0}\|_{H_h^1} + C(2\tau)^{1-\frac{p-1}{q_*}}\|u_h\|_{L_t^{q_*}(I; L_h^{\infty})}^{p-1}\|u_h\|_{C_t(I; H_h^1)}\\
&\le 2c_0\|u_{h,0}\|_{H_h^1} + C(2\tau)^{1-\frac{p-1}{q_*}}\left(4c_0\|u_{h,0}\|_{H_h^1}\right)^p.
\end{aligned}$$
Thus, it follows that 
\begin{equation}\label{L^infty bound proof 3}
\|u_h\|_{L_t^{q_*}(I; L_h^{\infty})}\leq 4c_0\|u_{h,0}\|_{H_h^1}
\end{equation}
as long as $C(2\tau)^{1-\frac{p-1}{q_*}}(4c_0\|u_{h,0}\|_{H_h^1})^p\leq 2c_0\|u_{h,0}\|_{H_h^1}$ is satisfied. Therefore, the time interval $I$ can be extended to a short time interval of which the length depends on $\|u_{h,0}\|_{H_h^1}$ but is independent of $h>0$.

To extend the time interval arbitrarily, we show that $\|u_h(t)\|_{H_h^1}$ is bounded uniformly in time. Indeed, by the mass conservation law, it is sufficient to show that $\|(-\Delta_h)^{\frac{1}{2}}u_h(t)\|_{L_h^2}$ is bounded globally in time. When $\lambda>0$, the energy conservation law immediately implies that $\|(-\Delta_h)^{\frac{1}{2}}u_h(t)\|_{L_h^2}^2\leq 2E_h(u_h(t))=2E_h(u_{h,0})$ for all $t$. When $\lambda<0$, we apply both the mass and the energy conservation laws as well as the 2D uniform Gagliardo--Nirenberg inequality (Lemma \ref{Lem:GN}) to obtain
$$\begin{aligned}
\frac{1}{2}\|(-\Delta_h)^{\frac{1}{2}}u_h(t)\|_{L_h^2}^2&=E_h(u_h(t))+\frac{\lambda}{p+1}\|u_{h}(t)\|_{L_h^{p+1}}^{p+1}\\
&\leq E_h(u_h(t))+C\|u_h(t)\|_{L_h^2}^{2}\|(-\Delta_h)^{\frac{1}{2}}u_h(t)\|_{L_h^2}^{p-1}\\
&\leq E_h(u_{h,0})+CM_h(u_{h,0})\|(-\Delta_h)^{\frac{1}{2}}u_h(t)\|_{L_h^2}^{p-1}.
\end{aligned}$$
By the assumption \eqref{assumption 1}, we have $p-1<2$. Thus, we can use Young's inequality to bound $\|(-\Delta_h)^{\frac{1}{2}}u_h(t)\|_{L_h^2}^2$ only in terms of the mass $M_h(u_{h,0})$ and the energy $E_h(u_{h,0})$.

Because $\|u_h(t)\|_{H_h^1}$ is bounded uniformly in time, \eqref{L^infty bound proof 3} can be iterated with the new initial data $u(\tau), u(2\tau), ...$ and with the bounds \eqref{L^infty bound proof 3} on the intervals $[\tau,2\tau]$, $[2\tau, 3\tau]$, ... to cover an arbitrarily long time interval $[-T,T]$. Therefore, summing up, we obtain the desired bound \eqref{long time L^infty bound}.
\end{proof}

\section{Proof of the contimuum limit}\label{sec: proof of main theorem}
In this section, we prove the main theorem of this article (Theorem \ref{main theorem}).

\subsection{Preliminaries}

We first provide lemmas concerning the discretization and linear interpolation (see \eqref{discretization} and \eqref{p_h}). Analogous lemmas on the lattice $h\mathbb{Z}^d$ have been stated and proven in \cite{HY2}. Thus, we omit some details. Indeed, differentiation (resp., discrete differentiation) is a local operation, thus the argument used in the non-compact domain $\R^d$ (resp, $h\Z^d$) can easily be adopted to the compact domain $\T^d$ (resp, $\T_h^d$).

\begin{lemma}[Boundedness of discretization and linear interpolation]\label{lem:pre1}
\[\|d_h(f)\|_{H_h^1(\mathbb{T}_h^d)} \lesssim \|f\|_{H^1(\mathbb{T}^d)} \quad \mbox{and} \quad \|p_h(f_h)\|_{H^1(\mathbb{T}^d)} \lesssim \|f_h\|_{H_h^1(\mathbb{T}_h^d)}.\]
\end{lemma}
\begin{proof}
We compute the discrete Sobolev norm using Lemma~\ref{Lem:normequivalence}.
Then the proof follows from the same method as \cite[Lemmas 5.1 and 5.2]{HY2}]
\end{proof}

\begin{lemma}\label{lem:pre2}
Let $h\in(0,1]$. Then, for $f\in H^1(\T^d)$, we have
\[\|(p_h\circ d_h)f - f\|_{L^2(\T^d)} \lesssim h\|f\|_{H^1(\T^d)}.\]
\end{lemma}

\begin{proof}
The proof closely follows from the proof of \cite[Proposition 5.3]{HY2}.
\end{proof}

\begin{lemma}\label{lem:pre3}
Let $h\in(0,1]$. If $f \in H^1(\T^d)$ and $g_h \in H_h^1(\T_h^d)$, then 
\[\| p_h e^{it\Delta_h}g_h - e^{it\Delta}f\|_{L^2(\mathbb{T}^d)} \lesssim \sqrt{h}|t|(\|g_h\|_{H_h^1(\mathbb{T}_h^d)} + \|f\|_{H^1(\mathbb{T}^d)}) + \|p_hg_h - f\|_{L^2(\mathbb{T}^d)}.\]
In particular,
\[\| p_h e^{it\Delta_h}d_h(f) - e^{it\Delta}f\|_{L^2(\mathbb{T}^d)} \lesssim \sqrt{h}\bra{t}\|f\|_{H^1(\mathbb{T}^d)}.\]
\end{lemma}

\begin{proof}
The proof closely follows from the proof of \cite[Proposition 5.4]{HY2}. First, using direct calculations, we observe that the Fourier transform of the linear interpolation of a discrete function is given by 
\begin{align*}
\mathcal{F}_h(p_h f_h)( k )= \mathcal P_h(k)
( \widetilde{ \mathcal F_h } f_h ) (k)
,\quad  \forall k\in \Z^d
\end{align*}
where
$$\mathcal P_h(k)= \frac{1}{h^d}\int_{[0,h)^d} e^{-ix\cdot k} dx + \sum_{j=1}^d \frac{e^{ihk_j}-1}{h}\frac{1}{h^d}\int_{[0,h)^d} x_je^{-ix\cdot k} dx$$
and $\widetilde{ \mathcal F_h }$ denotes the $[-\frac{\pi}{h},\frac{\pi}{h})^d$-periodic extension of the discrete Fourier transform $\mathcal F_h$, precisely, 
$( \widetilde{ \mathcal F_h } f_h ) (k)
=(\mathcal F_h f_h)(k')$ for all $k\in k' +\frac{2\pi}{h}\Z^d$. We also observe that 
\begin{align*}
\big| e^{-it\frac{4}{h^2}\sum_{j=1}^d \sin^2(\frac{h k_j}{2})}
- e^{it |k|^2} \big| \ls |t| h^2 |k|^4, \quad k\in (\T_h^d)^*.
\end{align*}	
By these observations and Lemma \ref{lem:pre1} and \ref{lem:pre2}, one can proceed as in the proof of \cite[Proposition 5.4]{HY2}. Here, an $O(\sqrt{h})$-bound is obtained from the regularity gap between the norms on the left- and right-hand sides.
\end{proof}

As a corollary of Lemma \ref{lem:pre3}, we have the following.
\begin{corollary}\label{cor:pre1}
Let $h\in(0,1]$ and $p>1$. Then,
\begin{equation}\label{eq:I_2}
\left\|\left(p_h e^{i(t-s)\Delta_h} - e^{i(t-s)\Delta}p_h\right)\left(|u_h|^{p-1}u_h \right)(s)\right\|_{L^2(\mathbb{T}^d)} \lesssim \sqrt{h}|t-s|\|u_h\|_{L_h^{\infty}(\mathbb{T}_h^d)}^{p-1}\|u_h\|_{H_h^1(\mathbb{T}_h^d)}.
\end{equation}
\end{corollary}
\begin{proof}
An immediate application of Lemma \ref{lem:pre3} to the left-hand side of \eqref{eq:I_2} yields
\[\mbox{LHS of } \eqref{eq:I_2} \lesssim h^\frac12|t-s|\left(\||u_h|^{p-1}u_h\|_{H_h^1} + \|p_h(|u_h|^{p-1}u_h)\|_{H^1}\right).\]
Lemma \ref{lem:pre1} and H\"older's inequality control the right-hand side, and we thus obtain \eqref{eq:I_2}.
\end{proof}

\begin{lemma}[Proposition 5.7 in \cite{HY2}]\label{lem:pre4}
Let $h\in(0,1]$ and $p > 1$. Then,
\[\|p_h\left(|u_h|^{p-1}u_h\right) - |p_hu_h|^{p-1}p_hu_h\|_{L^2(\mathbb{T}^d)} \lesssim h\|u_h\|_{L_h^{\infty}(\mathbb{T}_h^d)}^{p-1}\|u_h\|_{H_h^1(\mathbb{T}_h^d)}.\]
\end{lemma}

We end this section with the following lemma:
\begin{lemma}\label{lem:pre5}
Let $h\in(0,1]$ and $p > 1$. Then,
\[\||p_hu_h|^{p-1}p_hu_h - |u|^{p-1}u\|_{L^2} \lesssim \left(\|u_h\|_{L_h^{\infty}} + \|u\|_{L^{\infty}}\right)^{p-1}\|p_hu_h - u\|_{L^2}.\]
\end{lemma}

\begin{proof}
It follows from the calculation \eqref{FTC} with $\alpha=p_hu_h$ and $\beta=u_h$.
\end{proof}

\subsection{Proof of continuum limit}
Now we are in a position to prove Theorem \ref{main theorem}. Because the proof closely follows from the argument presented in \cite[Section 6]{HY2}, we only sketch the outline.

Let $h\in(0,1]$ be fixed. Given initial data $u_0\in H^1(\T^2)$, let $u(t) \in C(\R;H^1(\T^2))$ be the global solution to NLS \eqref{NLS} (see Section \ref{Section B}). For the discretization $u_{h,0}=d_hu_{0}$, let $u_h(t)$ be the solution to DNLS \eqref{DNLS} with the initial data $u_{h,0}$ constructed in Section \ref{sec:GWP}.

Applying the linear interpolation operator to the Duhamel formula \eqref{eq:Duhamel}, we write 
\[p_h u_h(t) = p_h e^{it\Delta_h}u_{h,0}  - i\lambda \int_0^t p_h e^{i(t-s)\Delta_h}\left(|u_h|^{p-1}u_h \right)(s) \; ds.\]
Then, by direct calculations, the difference of $u$ and $p_h u$ can be expressed as 
\[\begin{aligned}
p_h u_h(t) - u(t) =&~{} p_h e^{it\Delta_h}u_{h,0} - e^{it\Delta}u_0 \\
& - i\lambda \int_0^t \left(p_h e^{i(t-s)\Delta_h} - e^{i(t-s)\Delta}p_h\right)\left(|u_h|^{p-1}u_h \right)(s) \; ds\\
& - i\lambda \int_0^t e^{i(t-s)\Delta}\left(p_h\left(|u_h|^{p-1}u_h\right) - \left(|p_hu_h|^{p-1}p_hu_h\right) \right)(s) \; ds\\
& - i\lambda \int_0^t e^{i(t-s)\Delta}\left(|p_hu_h|^{p-1}p_hu_h - |u|^{p-1}u \right)(s) \; ds\\
=:&~{} I_1 + I_2 + I_3 + I_4.
\end{aligned}\]
Lemma \ref{lem:pre3}, \ref{lem:pre4}, and \ref{lem:pre5} and Corollary \ref{cor:pre1} yield
$$\begin{aligned}
&\|p_h u_h(t) - u(t)\|_{L^2}\\
&\lesssim h^{\frac12} \bra{t}^2(1+\|u_0\|_{H^1})^p + \int_0^t (\|u_h(s)\|_{L_h^{\infty}}^{p-1} + \|u(s)\|_{L_x^{\infty}}^{p-1})\|p_hu_h (s)- u(s)\|_{H^1} \; ds,
\end{aligned}$$
which, by applying Gr\"onwall's inequality in addition to Proposition \ref{L^infty bound} and Corollary \ref{cor:L^infty}, implies
\[\|p_h u_h(t) - u(t)\|_{L^2} \lesssim h^{\frac12}(1+\|u_0\|_{H^1})^p e^{B|t|}\]
for sufficiently large $B \gg 1$. This completes the proof of Theorem \ref{main theorem}. 

\appendix

\section{Proof of Lemma \ref{Sobolev inequality} and \ref{Lem:GN} }\label{sec: appendix Sobolev}

On a periodic domain, the proof of the Sobolev inequality is more involved, compared to that on the entire Euclidean space, because the explicit kernel formula for the inverse Laplacian is no longer available (see \cite{BO} for example). However, if an arbitrarily small loss of regularity is allowed, one can show the inequality in a simpler manner, as is presented in this appendix.

The key item is Bernstein's inequality for the projection operator $P_N$ (see \eqref{LP projection}).

\begin{lemma}[Bernstein's inequality]\label{Bernstein inequality}
Suppose that $0<s\leq\frac{d}{2}$, $q\geq 2$ and $\frac{1}{q}=\frac{1}{2}-\frac{s}{d}$. For $h\in(0,1]$ and a dyadic number $N$ with $N_*:=2^{\lceil\log_2(\tfrac{h}{\pi})\rceil-1}\leq N\leq 1$, we have
\begin{equation}\label{eq: Bernstein inequality}
\|P_{N}u\|_{L_h^q}\lesssim \left(\frac{N}{h}\right)^s\|u\|_{L_h^2}.
\end{equation}
\end{lemma}

\begin{proof}
We prove the lemma by the standard $TT^*$ argument. When $q=\infty$, we have
$$\begin{aligned}
\|P_{N}u\|_{L_h^\infty}&=\left\|\frac{1}{(2\pi)^d}\sum_{\frac{N\pi}{2h}<\max|k_j|\leq\frac{N\pi}{h}}(\mathcal{F}_hu)(k)e^{ik\cdot x}\right\|_{L_h^\infty}\\
&\lesssim \left(\frac{N}{h}\right)^{d}\|\mathcal{F}_hu\|_{L_k^\infty}\lesssim\left(\frac{N}{h}\right)^{d}\|u\|_{L_h^1}.
\end{aligned}$$
When $q=2$, it is obvious that $\|P_{N}u\|_{L_h^2}\leq \|u\|_{L_h^2}$. Interpolating, we obtain
$$\|P_{N}u\|_{L_h^q}\lesssim \left(\frac{N}{h}\right)^{2s}\|u\|_{L_h^{q'}}$$
for $q\geq 2$. This inequality implies that
$$\begin{aligned}
\|P_{N}u\|_{L_h^2}^2&=h^d\sum_{x}P_Nu(x)\overline{P_Nu(x)}=h^d\sum_{x}P_Nu(x)\overline{u(x)}\\
&\leq \|P_Nu\|_{L_h^q}\|u\|_{L_h^{q'}}\lesssim \left(\frac{N}{h}\right)^{2s}\|u\|_{L_h^{q'}}^2.
\end{aligned}$$
Thus, \eqref{eq: Bernstein inequality} follows from the duality.
\end{proof}

\begin{proof}[Proof of Lemma \ref{Sobolev inequality}]
By the triangle inequality and Lemma \ref{Bernstein inequality}, we prove that
$$\begin{aligned}
\|u\|_{L_h^q}&\leq\sum_{N=N_*}^1\|P_Nu\|_{L_h^q}\lesssim\sum_{N=N_*}^1\left(\frac{N}{h}\right)^s\|P_Nu\|_{L_h^2}\\
&\lesssim  \sum_{N=N_*}^1\left(\frac{N}{h}\right)^{-\epsilon}\|u\|_{H_h^{s+\epsilon}}\sim \left(\frac{N_*}{h}\right)^{-\epsilon}\|u\|_{H_h^{s+\epsilon}}\sim \|u\|_{H_h^{s+\epsilon}},
\end{aligned}$$
where in the last step, we used that $N_*\sim 2^{\log_2(\frac{h}{\pi})}=\frac{h}{\pi}$.
\end{proof}

Similarly, the Gagliardo--Nirenberg inequality can be proved.

\begin{proof}[Proof of Lemma~\ref{Lem:GN}]
Replacing $f$ by $\frac{1}{\|f\|_{L_h^2}}f$, we may assume that $\|f\|_{L_h^2}=1$. Suppose that $\|f\|_{H_h^{1}}\leq h^{-1}$. Let $R=h\|f\|_{H_h^{1}}$. Then, using Bernstein's inequality, we prove that
	\begin{align*}
	\|f\|_{L_h^q}
	&\leq \sum_{N_*\le N\le 1}\|P_Nf\|_{L_h^q}
	\leq \sum_{N_*\le N\le 1} \left(\frac{N}{h}\right)^\theta \|P_Nf\|_{L_h^2}\\
	&\leq
	 \sum_{N_*\leq N\leq R}(\frac{N}{h})^\theta \|P_Nf\|_{L_h^2}
	 +\sum_{R<N\leq 1}\left(\frac{N}{h}\right)^{\theta-1}\|P_N( (-\Delta_h)^\frac12f)\|_{L_h^2}\\
	 &\ls \left(\frac{R}{h}\right)^{\theta}+\left(\frac{R}{h}\right)^{\theta-1}\|f\|_{H_h^{1}} \sim\|f\|_{H_h^{1}}^{\theta}.
	\end{align*}
Similarly, if $\|f\|_{H_h^{1}}\geq h^{-1}$, then 
	$$\|f\|_{L_h^q}\leq \sum_{N_*\leq N\leq 1}\|P_Nf\|_{L_h^q}\lesssim\sum_{N_*\leq N\leq 1} \left(\frac{N}{h}\right)^{\theta}\|f\|_{L_h^2}\sim h^{-\theta}\le \|f\|_{H_h^{1}}^\theta.$$
\end{proof}

\section{Well-posedness results for NLS on the $\mathbb T^2$}\label{Section B}

We consider the (periodic) NLS \eqref{NLS}
\begin{equation}\label{eq:NLS}
\begin{aligned}
i\partial_t u+\Delta u-\lambda |u|^{p-1}u&=0,\\
u(0) &= u_0 \in H^s(\T^d).
\end{aligned}
\end{equation}
Duhamel's principle yields that \eqref{eq:NLS} is equivalent to the following integral equation on $[-T,T]$
\begin{equation}\label{eq:DUHAMEL_NLS}
u(t)=\eta_T(t)e^{-it(-\Delta)}u_0-i\lambda\eta_T(t)\int_0^t e^{-i(t-s)(-\Delta)}(|\eta_{2T}(t)u|^{p-1}\eta_{2T}(t)u)(s)ds,
\end{equation}
where $\eta$ is a smooth (even) bump function satisfying $\eta \equiv 1$ in $[-1,1]$ and $\eta \equiv 0$ in $(-2,2)^c$, and $\eta_T(t) = \eta(t/T)$. Note that one may replace $\eta_T(t)$ by $\eta(t)$ in \eqref{eq:DUHAMEL_NLS} (with a smallness assumption) when $p < 1+ \frac4d$ (in the $2D$ case, $p < 3$), owing to the scaling argument.

\medskip
 
For the classical well-posedness result of Bourgain \cite{B-93} (see also \cite{B-99}), we introduce the following function space. For $s,b\in \mathbb{R}$, we define the norm 
\[\norm{f}_{X^{s,b}}^2 = \int_{\R} \sum_{k \in \Z^d} \bra{k}^{2s}\bra{\tau + |k|^2}^{2b}|\widetilde{f}(\tau,k)|^2 d\tau\]
for $f \in \mathcal S (\R \times \T^d) $, where $\bra{\cdot} = (1+|\cdot|^2)^{1/2}$ and $\tilde{f}$ is the spacetime Fourier transform of $f$ given by
\[\widetilde{f}(\tau , k)=\int_{\mathbb{R}}\int _{\T^d} f(t,x) e^{-ix \cdot k}e^{-it\tau} dxdt.\]
Then, the $X^{s,b}$ space is defined as the completion of $\mathcal S'(\R \times \T^d)$ under the norm $\|\cdot\|_{X^{s,b}}$. This function space is termed the Bourgain space or the dispersive Sobolev space. 

\begin{theorem}[GWP for 2D NLS \cite{B-93}]\label{WP:NLS}
Suppose that $d=2$, and $p$ is given by \eqref{assumption 1}. Then, NLS \eqref{eq:NLS} is globally well-posed in $H^1(\T^2)$. Moreover, the solution $u$ obeys
\begin{equation}\label{eq:X}
\|u\|_{X^{1,\frac12}} \lesssim \|u_0\|_{H^1}.
\end{equation}
As a consequence, we have 
\begin{equation}\label{eq:Lp}
\|u\|_{L_{t,x}^q(\R \times \T^2)} \lesssim \|u\|_{X^{s(q,\epsilon),b(q,\epsilon)}}
\end{equation}
for $q\geq 4$, where $0 < \epsilon \ll 1$, $s(q,\epsilon) = \frac{4\epsilon}{q} + (1 + \frac{1-2\epsilon}{q-4})(1-\frac4q)$ and $b(q,\epsilon) = (\frac12 - \frac{\epsilon}{8})\frac4q + (\frac12 + \frac{\epsilon}{4(q-4)})(1-\frac4q)$. In particular,
\begin{equation}\label{eq:L4}
\|u\|_{L_{t,x}^4(\R \times \T^2)} \lesssim \|u\|_{X^{\epsilon,\frac12-\frac{\epsilon}{8}}}.
\end{equation}
\end{theorem}

\begin{remark}
$(i)$ One can immediately check $s(q,\epsilon) < 1 - \frac2q$ and $b(q,\epsilon) < \frac12$. \\
$(ii)$ In the one-dimensional case, Bourgain \cite{B-93} proved the $L_{t,x}^4$ estimate
\begin{equation}\label{eq:L4-1}
\|u\|_{L_{t,x}^4(\T \times \T)} \lesssim \|u\|_{X^{0,\frac38}}.
\end{equation}
This is an improvement of the $L^4$ estimate for free solutions by Zygmund \cite{Zygmund74}, namely,
\[\|e^{it\partial_x}u_0\|_{L_{t,x}^4(\T \times \T)} \lesssim \|u_0\|_{L^2},\]
which implies by the transference principle that
\[\|u\|_{L_{t,x}^4(\R \times \T)} \lesssim \|u\|_{X^{0,b}}, \quad b > \frac12.\]
$(iii)$ Bourgain employed a time-periodic function to show \eqref{eq:L4-1}; however, such a restriction is not necessary (such as \eqref{eq:L4}), see, for instance, \cite{Tao2001, Tzvetkov2006}.
\end{remark}

\begin{remark}
The $L^q$ estimate \eqref{eq:Lp} follows from the interpolation between \eqref{eq:L4} and $\|u\|_{L_{t,x}^{\infty}(\R \times \T^2)} \lesssim \|u\|_{X^{1+,\frac12+}}$.
Together with the H\"older inequality and the $L^4$ estimate \eqref{eq:L4}, one has the (local-in-time) $L^q$ estimate for $1 \le q <4$, precisely,
\[\|u\|_{L_{t,x}^q([0,1] \times \mathbb{T}^2)} \le \|u\|_{L_{t,x}^4(\R \times \T^2)} \lesssim \|u\|_{X^{\epsilon,\frac12-\frac{\epsilon}{8}}}\]  
\end{remark}

\begin{remark}
The \emph{a priori} bound \eqref{eq:X} can be obtained by the standard iteration method in addition to the $L^q$ estimate \eqref{eq:Lp}.
\end{remark}


As a corollary, we obtain a time-averaged bound.
\begin{corollary}[Time-averaged $L^{\infty}$ bound for 2D NLS]\label{cor:L^infty}
Suppose that $d=2$, and $p$ is given by \eqref{assumption 1}. Suppose that $u(t)$ is the global solution to periodic NLS \eqref{eq:NLS} with initial data $u_0$, constructed in Theorem \ref{WP:NLS}. Then,
\[\|u\|_{L^{q_*}([-T,T]; L^{\infty}(\T^2))} \lesssim \la T \ra^{\frac{1}{q_*}},\]
where $q_* > \max(p-1,2)$.
\end{corollary}

\begin{proof}
The proof follows from an analogous argument in the proof of Proposition \ref{L^infty bound}.
\end{proof}

%

\end{document}